\documentclass{amsart}
\usepackage{setspace, amsmath, amsthm, amssymb, amsfonts, amscd, epic, graphicx, ulem, dsfont}
\usepackage[T1]{fontenc}
\usepackage{multirow}
\usepackage{bbm}

\newtheorem{lem}{Lemma}
\newtheorem{thm}{Theorem}
\newtheorem*{un-thm}{Theorem}

\newtheorem{pro}{Proposition}

\theoremstyle{abstract}

\theoremstyle{remark}
\newtheorem*{rema}{Remark}

\newtheorem*{exa}{Example}

\theoremstyle{definition}

\newcommand{\R}{\mathbb{R}}

\makeatletter \@namedef{subjclassname@2010}{
  \textup{2010} Mathematics Subject Classification}
\makeatother

\begin{document}

\title[A Criterion for the Normality of Unbounded Operators]{A Criterion for the Normality of Unbounded Operators and Applications to Self-adjointness}
\author{Mohammed Hichem Mortad}

\address{Department of
Mathematics, University of Oran, B.P. 1524, El Menouar, Oran 31000.
Algeria.\newline {\bf Mailing address}:
\newline Dr Mohammed Hichem Mortad \newline BP 7085
Es-Seddikia\newline Oran
\newline 31013 \newline Algeria}

\email{mhmortad@gmail.com, mortad@univ-oran.dz.}

\begin{abstract}
In this paper we give and prove a criterion for the normality of
unbounded closed operators, which is a sort of a maximality result
which will be called "double maximality". As applications, we show,
under some assumptions, that the sum of two symmetric operators is
essentially self-adjoint; and that the sum of two unbounded normal
operators is essentially normal. Some other important results are
also established.

\end{abstract}



\thanks{Partially supported by "Laboratoire d'Analyse Mathématique et Applications".}

\maketitle

\section{Introduction}

First, we assume the reader is familiar with notions and results
about bounded and unbounded linear operators. Some general
references are \cite{Con,GGK,Kato-Book,Putnam-book,RUD,SCHMUDG-book-2012,WEI}.

Let us, however, recall the next known result which will be useful.

\begin{lem}[\cite{WEI}]\label{(AB)*=B*A*} If $A$ and $B$ are densely defined and $A$ is invertible with
inverse $A^{-1}$ in $B(H)$, then $(BA)^* =A^* B^*$.
\end{lem}

In order to avoid an eventual confusion for Banach algebraists, we
say here that an unbounded densely defined closeable operator $A$ is
essentially normal if its closure $\overline{A}$ is normal.

It is known that self-adjoint operators are maximally symmetric.
Hence, any closed operator $A$ such that $AA^*\subset A^*A$ is
automatically normal for $AA^*$ and $A^*A$ are both self-adjoint. It
is also well-known that normal operators are maximally normal. For
some other maximality results, see \cite{DevNussbaum-von-Neumann}
and \cite{Nussbaum-TAMS-commu-unbounded-normal-1969}.

However, there are some situations which one encounters sometimes
and which are not covered by any of the results cited above. For
instance what can we say about a closed or non closed densely
defined operator $S$ obeying $T\subset SS^*$ and $T\subset S^*S$
(where $T$ is just densely defined)? The answer will be given in
Theorem \ref{Main Theorem crucial S normal } below. Then we prove a
more general result. It mainly gives conditions on when $T\subset R$
and $T\subset S$ imply that $S=R$? See Theorem \ref{maximality
double S,T,R}.

As an important consequence of Theorem \ref{Main Theorem crucial S
normal }, we establish a new result on the essential
self-adjointness of the sum of two unbounded symmetric operators
(see Theorem \ref{essential self-adjointness A+B}).

Then we prove a result on the normality of the sum of two unbounded
normal operators (cf \cite{mortad-CAOT-sum-normal}). A certain form
of commutativity (not strong commutativity though) is required.

For the reader's convenience, let us gather in one theorem some of
the famous results on the self-adjointness of the sum of two
self-adjoint:

\begin{thm}\label{self-adjointness (essential s.a.) A+B all results}
Let $A$ and $B$ be two operators with domains $D(A)$ and $D(B)$
respectively, such that $A+B$ is densely defined \footnote{ Kosaki
\cite{KOS} gave explicit examples of unbounded densely defined
self-adjoint positive operators $A$ and $B$ such that $D(A)\cap
D(B)=\{0\}$.}. Then $A+B$ is:
\begin{enumerate}
  \item self-adjoint on $D(A)$ if $A$ and $B$ are self-adjoint, and
  if $B$ is bounded.
  \item self-adjoint on $D(A)\cap D(B)$ whenever $A$ and $B$ are
  commuting self-adjoint and positive operators (see \cite{Putnam-book}).
  \item self-adjoint on $D(A)\cap D(B)$ whenever $A$ and $B$ are
  anticommuting self-adjoint operators (see \cite{Vasilescu-1983-anti-commuting}).
  \item (Kato-Rellich) self-adjoint on $D(A)$ if $B$ is symmetric and $A$-bounded with relative bound
  $a<1$, and $A$ is self-adjoint (see \cite{RS2}).
\end{enumerate}
\end{thm}

For papers treating similar questions on sums, see
\cite{Mortad-CMB-2011,mortad-CAOT-sum-normal,Vasilescu-1983-anti-commuting}.

\section{Main Results}

The first important result of the paper is a "double maximality"
one:
\begin{thm}\label{Main Theorem crucial S normal }
Let $S,T$ be two densely defined unbounded operators on a Hilbert
space $H$ with respective domains $D(S)$ and $D(T)$. Assume that
\[\left\{\begin{array}{c}
                 T\subset S^*S, \\
                 T\subset SS^*,
               \end{array}
\right.\] Let $D\subset D(T)$ ($\subset D(S^*S)\cap D(SS^*)$) be
dense.
\begin{enumerate}
  \item Assume that $D$ is a core, for example, for $S^*S$. If $S$ is closed, then $S$ is
normal.
  \item  If $S$ is not closed, then $\overline{S}$  is normal if $D$
is a core for $S^*\overline{S}$.
\end{enumerate}
\end{thm}

\begin{proof}\hfill
\begin{enumerate}
  \item Recall that $SS^*$ and $S^*S$ are both self-adjoint on the
respective domain $D(SS^*)$ and $D(S^*S)$. We clearly have
\[0_{D(T)}=T-T\subset S^*S-SS^*.\]
Hence
\[S^*S-SS^*\subset (S^*S-SS^*)^*\subset [0_{D(T)}]^*=0_H\]
If we denote the restriction of $S^*S$ to $D$ by $S^*S_D$, then
\begin{equation}\label{eq 1}
S^*S_{D}\subset S^*S-SS^*+SS^*\subset SS^*.\end{equation}

Since $D$ is a core for $S^*S$, we have
\[S^*S=\overline{S^*S_{D}}.\]
Since self-adjoint operators are maximally symmetric, passing to
closures in the inclusions (\ref{eq 1}) we see that we have

\[S^*S=SS^*,\] proving the normality of
$S$.
  \item If $S$ is not closed, the proof is very similar. We just observe that $S^*$ is densely
defined and that in this case
\[S^*\overline{S}-\overline{S}S^*\subset (S^*S-SS^*)^*\subset [0_{D(T)}]^*=0_H.\]
\end{enumerate}
\end{proof}

\begin{rema}
Of course, if we only have

\[\left\{\begin{array}{c}
                 T\subset S^*S, \\
                 T\subset SS^*,
               \end{array}
\right.\] (i.e. without any "core condition"), then $S$ is not
necessarily normal even if it is closed as seen in the following
example:
\end{rema}
\begin{exa}\label{Example Sandy}
Let $S$ be defined by $Sf(x)=-f'(x)$ on $D(S)=\{f\in L^2(0,1):~f'\in
L^2(0,1)\}$. Then $S$ is densely defined and closed, but it is not
normal. Indeed,
\[\text{$S^*f(x)=-f'(x)$ on $D(S^*)=\{f\in L^2(0,1):~f'\in
L^2(0,1),~f(0)=f(1)=0\}$}\] so that
\[SS^*f(x)=S^*Sf(x)=-f''(x)\]
with
\[D(SS^*)=\{f\in L^2(0,1):~f''\in
L^2(0,1),~f(0)=f(1)=0\}\] and
\[D(S^*S)=\{f\in L^2(0,1):~f''\in
L^2(0,1),~f'(0)=f'(1)=0\}.\] Hence $S$ is not normal. Let $T$ be
defined by $Tf(x)=-f''(x)$ on
\[D(T)=\{f\in L^2(0,1):~f''\in
L^2(0,1),~f(0)=f(1)=0=f'(0)=f'(1)\}.\] Then
\[\left\{\begin{array}{c}
                 T\subset S^*S, \\
                 T\subset SS^*,
               \end{array}
\right.\] $T$ and $S$ are densely defined, $S$ is closed, and yet as
we have just seen $S$ is not normal.
\end{exa}

Adopting the same proof as that of Theorem \ref{Main Theorem crucial S normal }, we can easily generalize the latter theorem to the following "double maximality" result which is, by the way, stronger than Theorem 5.31 in \cite{WEI}
(in \cite{WEI}, the coming $T$ was assumed to be symmetric and it was also assumed that $D(R)\subset D(S)$).
We leave the proof to the reader.

\begin{thm}\label{maximality double S,T,R}
Let $R,S,T$ be three densely defined unbounded operators on a
Hilbert space $H$ with respective domains $D(R)$, $D(S)$ and $D(T)$.
Assume that
\[\left\{\begin{array}{c}
                 T\subset R, \\
                 T\subset S.
               \end{array}
\right.\] Assume further that $R$ and $S$ are self-adjoint. Let
$D\subset D(T)$ ($\subset D(R)\cap D(S)$) be dense. Let $D$ be a
core, for instance, for $S$. Then $R=S$.
\end{thm}

\begin{rema}
After I proved Theorem \ref{Main Theorem crucial S normal }, I came
across a result  due to Stochel-Szafraniec
\cite{Stochel-Szafraniec-2003-JMSJ-domination-commutativity} which I
was not aware of. The authors proved a similar result to mine. They
further assumed that $S^*S_D$ was essentially self-adjoint
(something obtained in my proof) but then they obtained in their
proof that $D$ is a core for $S^*S$ (something assumed in my
theorem). In fact, it is well-known that a linear subspace $D$ is core for a self-adjoint operator $A$ iff $A$ restricted to $D$ is essentially self-adjoint.
\end{rema}

Now we prove an interesting result on the sum of two symmetric
operators. This is the natural generalization of the known result on
the sum of two \textit{bounded} symmetric (or self-adjoint)
operators.

\begin{thm}\label{essential self-adjointness A+B}
Let $A$ and $B$ be two unbounded symmetric operators. Let $D$ be a
dense linear manifold contained in the domain $D[(A+B)^2]$. If $D$
is a core for $(A+B)^*(\overline{A+B})$, then $A+B$ is essentially
self-adjoint on $D$.
\end{thm}

\begin{rema}
The hypothesis $A$ and $B$ being symmetric is weak but the
assumption
 $(A+B)^2$ being densely defined is strong (where $A+B$ is symmetric).

Remember that Chernoff \cite{CH} (see also \cite{DIX,NAI,SCHMUDG})
gave an explicit example of a symmetric, closed and semi-bounded
operator $A$ satisfying $D(A^2)=\{0\}$ (hence $D(A^2)$ is far from
being dense!).

For further investigations on when $D(A^n)$ is dense, see the
intensive work \cite{SCHMUDG}.
\end{rema}

\begin{rema}
In practise, especially when dealing with partial differential
operators, it is usually fairly simple to find the appropriate dense
domain $D$. In this case, the domain par excellence is
$C_0^{\infty}(\R^n)$.
\end{rema}

We may exploit Chernoff's counterexample to show the importance of
the dense manifold $D$ in the foregoing theorem as seen in
\begin{exa}
Let $A$ be a closed and symmetric operator such that $D(A^2)=\{0\}$.
Hence $A$ is not self-adjoint (if it were, then $A^2$ would be
self-adjoint too and $D(A^2)$ would then be dense). Set $B=A$, then
\[D[(A+B)^2]=D(4A^2)=D(A^2)=\{0\}.\]
Hence $D[(A+B)^2]$ does not contain any dense set $D$. Finally, as
observed above,
\[\overline{A+B}=\overline{A+A}=2A\] is not
self-adjoint.
\end{exa}

Now we prove Theorem \ref{essential self-adjointness A+B}.
\begin{proof}
Since $D\subset D[(A+B)^2]$, $(A+B)^2$ is densely defined, and then
so is $A+B$. Hence $A+B$ is closeable for
\[A+B\subset A^*+B^*\subset (A+B)^*.\]
Then we have
\[A+B\subset (A+B)^*=(\overline{A+B})^*\Longrightarrow \overline{A+B}\subset (\overline{A+B})^*,\]
that is $\overline{A+B}$ is symmetric.  We also have
\[A+B\subset (A+B)^*\Longrightarrow
(A+B)^2\subset (A+B)^*(A+B).\]

Similarly,

\[(A+B)^2\subset (A+B)(A+B)^*.\]

Hence Theorem \ref{Main Theorem crucial S normal } implies that
$\overline{A+B}$ is normal. Since symmetric normal operators are
self-adjoint, we immediately deduce that $A+B$ is essentially
self-adjoint on $D$.
\end{proof}

We have the analog of Theorem \ref{essential self-adjointness A+B}
for a finite family of operators. The proof is left to the
interested reader.

\begin{thm}\label{essential self-adjointness A1+A2+...+An}
Let $(A_k)_{1\leq k\leq n}$ be a family of symmetric operators. Let
$D$ be a dense linear manifold contained in the domain
$D[(A_1+\cdots+A_n)^2]$. If $D$ is a core for
$(A_1+\cdots+A_n)^*(\overline{A_1+\cdots+A_n})$, then
$A_1+\cdots+A_n$ is essentially self-adjoint on $D$.
\end{thm}

It is a classic matter to take an unbounded normal operator $A$ and
then set $B=-A$ to see that $A+B$ is not normal (as it is not
closed). But in this case $A+B$ is essentially normal for its
closure is worth $"0"$ defined everywhere. This observation has
inspired me to give and prove the following result:

\begin{pro}
\label{normal of the closure A+B BETTER} Let $A$ and $B$ be two
unbounded normal operators obeying $AB^*=B^*A$. Assume that only $B$
is invertible. Let $D$ be a dense linear manifold contained in the
domains of $A^*A$, $B^*A$, $A^*B$ and $B^*B$. Then $A+B$ is
essentially normal on $D$ if $D$ is a core for $(A+B)^*(A+B)$.
\end{pro}

\begin{proof}
We have
\begin{align*}
(A+B)(A+B)^*&\supset (A+B)(A^*+B^*)\\
&=A(A^*+B^*)+B(A^*+B^*)\\
&\supset AA^*+AB^*+BA^*+BB^*.
\end{align*}

Similarly, we obtain
\[(A+B)^*(A+B)\supset A^*A+B^*A+A^*B+B^*B.\]

Since $AB^*=B^*A$ and $B$ is invertible, Lemma \ref{(AB)*=B*A*}
yields $A^*B\subset BA^*$. Hence
\[A^*A+B^*A+A^*B+B^*B\subset AA^*+AB^*+BA^*+BB^*.\]
Set $Q=A^*A+B^*A+A^*B+B^*B$. Then
\[ (A+B)(A+B)^*\supset Q\text{ and } (A+B)^*(A+B)\supset Q.\]
Since $Q$ is densely defined, Theorem
\ref{Main Theorem crucial S normal } yields the desired result.
\end{proof}

The last result of the paper concerns the normality of the sum too. It is well-known that the sum of two bounded commuting normal
operators is normal. This was generalized to the case where one
operator is unbounded (see \cite{mortad-CAOT-sum-normal}). It is
also known that the sum of two strongly anti-commuting unbounded
self-adjoint operators is self-adjoint (see
\cite{Vasilescu-1983-anti-commuting}). This result, however, cannot
be generalized to the case of two strongly anticommuting normal
operators even for the case of two bounded normal operators. Indeed,
let
\[A=\left(
      \begin{array}{cc}
        2 & 0 \\
        0 & -2 \\
      \end{array}
    \right)
 \text{ and } B=\left(
                  \begin{array}{cc}
                    0 & 1 \\
                    -1 & 0 \\
                  \end{array}
                \right)
 .\]
Then $A$ and $B$ are both normal. They anticommute because
\[AB=-BA=\left(
           \begin{array}{cc}
             0 & 2 \\
             2 & 0 \\
           \end{array}
         \right),
\]
but as one can easily check
\[A+B=\left(
      \begin{array}{cc}
        2 & 1 \\
        -1 & -2 \\
      \end{array}
    \right)\]
    is not normal.

Nevertheless, we have the following result:
\begin{thm}\label{normality sum anti-commuting B is BD}
Let $A$ be an unbounded normal operator and let $B$ be
bounded and self-adjoint. If $BA^*\subset -AB$, then $A+B$ is normal.
\end{thm}

The proof requires the following lemma whose proof is very akin to
the one in \cite{mortad-CAOT-sum-normal}.

\begin{lem}\label{normality sum bounded and unbounded domains!!!!}
Let $A$ be an unbounded normal operator with domain $D(A)$. Let $B$
be a bounded self-adjoint operator defined on a Hilbert space $\mathcal{H}$. If
$BA^*\subset -AB$, then
\[D[A^*(A+B)]=D(A^*A) \text{ and } D[A(A^*+B)]=D(AA^*),\]
so that
\[A^*(A+B)=A^*A+A^*B \text{ and } A(A^*+B)=AA^*+AB.\]
\end{lem}

Now we prove Theorem \ref{normality sum anti-commuting B is BD}

\begin{proof}
Since $B$ is bounded (and self-adjoint), we have
\[A+B \text{ is closed and } (A+B)^*=A^*+B^*~(=A^*+B).\]
Now thanks to Lemma \ref{normality sum bounded and unbounded
domains!!!!} we have
\[(A+B)^*(A+B)=A^*A+A^*B+BA+B^2.\]
and
\[(A+B)(A+B)^*=AA^*+AB+BA^*+B^2.\]
Since $BA^*\subset -AB$, we have $BA\subset -A^*B$ (by the Fuglede-Putnam theorem), \text{ so
that } \[BA^*+AB\subset 0 \text{ and $BA+A^*B\subset 0$}.\]
Hence
\[(A+B)(A+B)^*=AA^*+AB+BA^*+BB^2\subset AA^*+B^2\]
and
\[(A+B)^*(A+B)=A^*A+A^*B+BA+B^2\subset A^*A+B^2.\]
Since $AA^*$ is self-adjoint and $B^2$ is self-adjoint (and
bounded), $AA^*+B^2$ is self-adjoint too. Since $A+B$ is closed,
$(A+B)(A+B)^*$ is self-adjoint. But, self-adjoint operators are
maximally symmetric. Thus
\[(A+B)(A+B)^*=AA^*+B^2.\]
Similarly, we obtain
\[(A+B)^*(A+B)=A^*A+B^2.\]
Therefore, and by the normality of both $A$ and $B$, we get that
\[(A+B)(A+B)^*=(A+B)(A+B)^*,\]
completing the proof.
\end{proof}

Very similarly, we also have

\begin{thm}\label{normality sum anti-commuting B is BD}
Let $A$ and $B$ be two normal operators such that only $B$ is
bounded. If $BA^*\subset -AB^*$ and $B^*A\subset -A^*B$, then $A+B$ is normal.
\end{thm}


\begin{thebibliography}{1}

\bibitem{CH}
P. R.~Chernoff, A Semibounded Closed Symmetric Operator Whose Square
Has Trivial Domain, \textit{Proc. Amer. Math. Soc.,} {\bf
89/2}\textnormal{ (1983) 289-290}.

\bibitem{Con}
J. B. Conway, \textit{A Course in Functional Analysis}, \textnormal{
Springer, GTM {\bf 68} (1990), 2nd edition}.


\bibitem{DevNussbaum-von-Neumann}
A. Devinatz, A. E. Nussbaum, J. von Neumann, On the Permutability of
Self-adjoint Operators, \textit{Ann. of Math. (2)}, {\bf 62} (1955),
199-203.

\bibitem{DIX}
J.~Dixmier, \textnormal{ L'adjoint du Produit de Deux Opérateurs
Fermés,} \textit{Annales de la faculté des sciences de Toulouse
4ème Série}, {\bf 11}\textnormal{ (1947) 101-106}.

\bibitem{GGK}
I. Gohberg, S. Goldberg, M. A. Kaashoek, \textit{Basic Classes of
Linear Operators}, Birkh\"{a}user Verlag, Basel, 2003.

\bibitem{Kato-Book} T. Kato, {\it Perturbation Theory for Linear Operators}, 2nd
Edition, Springer, 1980.

\bibitem{KOS}
H. Kosaki, {\textit{On Intersections of Domains of Unbounded
Positive Operators},} Kyushu J. Math., {\bf 60/1} (2006) 3-25.

\bibitem{Mortad-CMB-2011}
M. H. Mortad, \textit{On the Adjoint and the Closure of the Sum of
Two Unbounded Operators}, Canad. Math. Bull., {\bf 54/3} (2011)
498-505. DOI:10.4153/CMB-2011-041-7.

\bibitem{mortad-CAOT-sum-normal}
M. H. Mortad, \textit{On the Normality of the Sum of Two Normal
Operators}, Complex Anal. Oper. Theory, {\bf 6/1} (2012), 105-112.
DOI: 10.1007/s11785-010-0072-7.

\bibitem{NAI}
M.~Naimark, \textnormal{On the Square of a Closed Symmetric
Operator,} \textit{Dokl. Akad. Nauk SSSR}, {\bf 26} \textnormal{
(1940) 866-870; ibid. {\bf 28} (1940), 207-208}.

\bibitem{Nussbaum-TAMS-commu-unbounded-normal-1969}
A. E. Nussbaum, \textit{A Commutativity Theorem for Unbounded
Operators in Hilbert Space}, Trans. Amer. Math. Soc. {\bf 140}
(1969), 485-491.

\bibitem{Putnam-book}C. Putnam, {\it Commutation Properties of Hilbert Space
Operators}, Springer, 1967.

\bibitem{RS2}
M. Reed, B. Simon, \textit{Methods of Modern Mathematical Physics},
Vol. {\bf 2}: Fourier Analysis, Self-Adjointness, Academic Press,
1975.

\bibitem{RUD}
W.~Rudin, \textit{Functional Analysis}, McGraw-Hill, 1991 (2nd
edition).

\bibitem{SCHMUDG}
K. Schm\"{u}dgen, \textit{On Domains of Powers of Closed Symmetric
Operators}, J. Operator Theory, {\bf 9/1} (1983) 53-75.


\bibitem{SCHMUDG-book-2012}
K. Schm\"{u}dgen, \textit{Unbounded Self-adjoint Operators on
Hilbert Space,} Springer GTM {\bf 265}  (2012).

\bibitem{Stochel-Szafraniec-2003-JMSJ-domination-commutativity}
J. Stochel, F. H. Szafraniec, \textit{Domination of unbounded
operators and commutativity}, J. Math. Soc. Japan \textbf{55/2}
(2003), 405-437.

\bibitem{Vasilescu-1983-anti-commuting}
F. H. Vasilescu, \textit{Anticommuting Selfadjoint Operators}, Rev.
Roumaine Math. Pures Appl. \textbf{28/1} (1983), 76-91.

\bibitem{WEI}
J.~Weidmann, Linear operators in Hilbert spaces (translated from the
German by J. Sz\"{u}cs), Srpinger-Verlag, GTM {\bf 68} (1980).
\end{thebibliography}
\end{document}